%% file: arxiv-submission.tex
\documentclass[envcountsame,envcountsect]{svmult}

\usepackage{type1cm}        
\usepackage{makeidx}         
\usepackage{graphicx}        
\usepackage{multicol}        
\usepackage[bottom]{footmisc}

\usepackage{todonotes}
\usepackage{tikz}

\usepackage{newtxtext}       %
\usepackage[varvw]{newtxmath}       
\newcommand{\eTPF}[2]{\mathrm{eTPF}(#1,#2)}
\newcommand{\mult}{\mathrm{mult}}

\makeindex             

\begin{document}

\title*{A colorful way to park: An introduction to exact $k$-Typed Parking Functions}
\author{Aalliyah Celestine\orcidID{0000-0002-9355-1383} and \\ Lina Liu\orcidID{0000-0002-0742-9032} and \\ Jacob van der Leeuw}
\institute{Aalliyah Celestine 
\and Lina Liu 
\and Jacob van der Leeuw 
}
%
%
\maketitle

\abstract{Parking functions are tuples that describe the parking of $M$ cars on a street with $M$ parking spots. In this paper, we define exact $k$-typed parking functions ($k$-TPFs) to be a variant of classical parking functions.
We then establish that every exact $k$-TPF $\alpha$ of length $M$, corresponds to a unique parking configuration $C$. We observe that the collection of all exact $k$-TPFs which result in the same configuration form a disjoint subset of all exact $k$-TPFs. Lastly, we conclude by showing how parking permutations of an exact $k$-TPF can be related to other combinatorial objects.}

\section{Introduction}
Parking functions were initially introduced by Konheim and Weiss \cite{science-DOI-kon} and are defined as follows: Imagine a one-way street with parking spots labeled with the integers $1$ through $M$. Suppose there are $M$ cars wanting to park, each of which has a parking preference. A \emph{parking preference} is in the form of the tuple $\alpha = (p_1, \ldots, p_M) \in [M]^M$, where $p_j$ is the preferred spot in which car $c_j$ would park. Thus, for $j=1,\ldots,M$, the driver of the $j^{th}$ car parks according to their preference in order from car $c_1$ to car $c_M$. If car $c_j$'s preferred spot $p_j$ is empty, then the car parks there. If the spot is instead occupied, car $c_j$ continues moving forward until it finds the next available spot. If the car reaches the end of the street without having parked, then $\alpha$ is not a valid parking function. Otherwise, all cars can park in the preferred $M$ number of spots, and $\alpha$ is a parking function. 

A variation of classical parking functions, that we will present, assigns a \textit{type} from 1 through $k$ to each of the cars and is known as an exact $k$-typed parking function. The allotted rules for parking are as follows: Suppose there is an empty one-way street with $M$ parking spots available. 
The preferences for each car now depends on how many cars of a lower type have 
been previously parked.
We use $P_i$ to denote the preference list associated with cars of type $i$.
The final parking function, which we denote $\alpha=(m_1; P_2, \ldots, P_k)$, is the tuple containing preferences for the cars grouped by type.
In this paper, we answer three major questions regarding exact $k$-typed parking functions.
\begin{enumerate}
\item How can we directly enumerate the number of exact $k$-typed parking functions, which we denote as $k$-TPFs? Here, we count the number of exact $k$-TPFs by considering the number of parking possibilities each car of each type has.
\item \label{two} How many $\alpha$-permutations of an exact $k$-TPF are there for a given number of type $k$? A one-to-one relationship is then established between parking configurations and the families of exact $k$-TPFs. We also demonstrate that these families are disjoint and span the entire space of exact $k$-TPFs. And lastly,
\item \label{three} How many exact $k$-typed parking functions are there?
\end{enumerate}

The results of the last two questions determine the total number of exact $k$-TPFs. Ultimately, this establishes an equality by showing that the number of exact $k$-TPFs can be counted in two distinct ways: first, by a classical counting argument (\ref{two}) and secondly, counting parking functions via groupings (\ref{three}).

We organized the paper as follows: In Section 2, we provide the definition of an exact $k$-typed parking function along with examples to clarify some basic properties of an exact $k$-TPF. We also define parking permutations for exact-$k$-TPFs in Section 2.
Following this, in Section 3 we present a way to count all the $k$-TPFs given a fixed $M$ and order ($m_1,m_2\ldots, m_k)$. In Section 4, we define a family of exact $k$-TPFs given the definition of parking permutations. We conclude with a theorem that connects the number of families to the number of distinct parking configurations. 
Section 5 provides an alternative way to count all the exact $k$-typed parking functions using the families defined in Section 4. 
Finally, we discuss a list of directions for future works in Section 6.

\section{Basic properties of $k$-TPFs}\label{sect2}

In this section we discuss the parking rules for $k$-typed parking functions. Let $M$ denote the number of parking spots and cars. Let $m_1, \ldots, m_k$ be a sequence of positive integers with $m_1 + \cdots + m_k = M$. Each $m_i$ corresponds to the number of cars of type $i$, for $i\in \{1,...,k\}$. 

\begin{definition}($k$-typed parking function).
We call a sequence $\alpha=(m_1; P_2, \ldots, P_k)$ a \emph{$k$-typed parking function ($k$-TPF)} of order $(m_1, m_2, \ldots, m_k)$ if there exists a parking configuration of the $M$ cars that satisfies the following for all drivers:
\begin{enumerate}
    \item Each $P_i$ corresponds to a preference list $P_i = \left(p_1^{(i)}, p_2^{(i)}, \ldots, p_{m_i}^{(i)} \right)$ for cars of type $i$, where $i \geq 2$.
    \item For each $p_j^{(i)}$, the superscript $(i)$ denotes the car's type and the subscript $j$ is the index within preference list $P_i$.
    \item For each preference list $P_i$, the driver of the $j^{\rm th}$ car has the preference that there be {\it at least} $p_j^{(i)}$ cars of type $\{1, \ldots, i-1\}$ parked before them. This driver's preference does not depend on other cars of type $i$.
\end{enumerate} 
\end{definition}

\begin{definition}(Exact $k$-typed parking functions). \label{def: exact kTPF}
    In the case that every car prefers \emph{exactly} $p_{j}^{(i)}$ cars of type $\{1, \ldots, i-1\}$ parked before them, then we say $\alpha$ is an \emph{exact $k$-typed parking function (exact $k$-TPF)} of order $(m_1, m_2, \ldots, m_k)$. We let $\eTPF{M}{k}$ denote the set of all exact $k$-typed parking functions of $M$ cars \footnote{When $k=1$, the $m_1$ cars of type 1 inherently do not have any preference list. That is, the cars park without restrictions and do not need to be considered when describing preference values.}.
\end{definition}

\begin{definition}(Resulting configuration).
    If $\alpha = (m_1; P_2, \ldots, P_k)$ is an exact $k$-typed parking function of order $(m_1, m_2, \ldots, m_k),$ then we call the outcome of its parking arrangement the \emph{resulting configuration} $C$.
\end{definition}

To illustrate the difference between the original $k$-TPF definition that uses ``at least'' versus the exact $k$-TPF definition, we look at the following example.
\begin{example}\label{example1}
    Consider $\alpha=(2; (0))$ as a 2-TPF. We observe that the only car of type 2 prefers to park before {\it at least} 0 cars of type 1. Hence, the car can either park in the first available spot or behind the first or second car of type 1. Therefore, the three possible resulting configurations are $211$, $121$, or $112$. If we instead use the definition of exact $k$-TPFs, the car of type 2 must park in the $0^{\text{th}}$ {\textit {gap}} and gives us the resulting configuration $211$. 
\end{example}

Throughout the rest of the paper, our results are pertaining to {\it exact} $k$-TPFs unless specified otherwise. We now discuss the motivation and definition of \emph{gap} as it is used in Example~\ref{example1}.

Given an exact $k$-TPF $\alpha=(m_1; P_2, \ldots, P_k)$, and when parking cars of type $i \geq 2$, there are $m_1+ m_2+ \cdots +m_{i-1}$ cars of a lower type available to be chosen as a preference.
\begin{definition}(Gap).
 We call the $\ell^{\text{th}}$ \emph{gap} the space after the $\ell^\text{th}$ parked car of type less than $i$. With the convention that the $0^\text{th}$ gap is the space before all cars of type less than $i$. 
\end{definition}
We note that gaps do not have a fixed size, meaning that we can fit up to $m_i$ cars of type $i$ in any  given gaps. Hence there is a distinction between the words ``space'' and  ``gap''.
\footnote{This explains why it is possible for multiple cars of type $i$ to share the same preference.}

We now provide two separate techniques of parking exact $k$-TPFs that lead to the same resulting configuration. Analyzing these configurations will ultimately uncover families of parking functions based on permuting the entries of the preference lists.

\begin{example}\label{example2}
    Let $\alpha = (4;(0,1,1,2,2))$ be an exact 2-typed parking function of order (4,5). One technique to finding the resulting parking configuration is to park one car at a time, regardless of repeated preferences within a type. First, place the cars of type 1 without restriction in a parking spot. Then the following tuple $(0,1,1,2,2)$ is parked as follows:
    \begin{itemize}
        \item The first car of type 2 parks in gap 0.
        \item The second car of type 2 parks in gap 1.
         \item The third car of type 2 parks in gap 1.
    \end{itemize}
    Repeat this until all the cars are parked. The resulting configuration is as follows:
\begin{figure}[h!!!]
    \centering
    \boxed{\textcolor{black}{2}}\boxed{\textcolor{black}{1}}\boxed{\textcolor{black}{2}}\boxed{\textcolor{black}{2}}\boxed{\textcolor{black}{1}}\boxed{\textcolor{black}{2}}\boxed{\textcolor{black}{2}}\boxed{\textcolor{black}{1}}\boxed{\textcolor{black}{1}}
       \label{fig:my_label}
\end{figure}
\end{example}

We now provide a second way to park cars involving simultaneous parking that results in the same configuration. 
\begin{example}\label{example3}
Again, let $\alpha = (4;(0,1,1,2,2))$. Here, there is only one preference list $P_2$ for which we park in three steps below:
\begin{itemize}
        \item Park the $1$ car of type $2$ with a preference of $0$ in gap 0.
        \item Park the $2$ cars of type $2$ with a preference of $1$ in gap $1$ simultaneously.
         \item Park the $2$ cars of type $2$ with a preference of $2$ in gap $2$ simultaneously.
\end{itemize}
\end{example}

We observe that iterative parking (Example 2) or simultaneously parking cars of the same preference (Example 3) will produce the same resulting configuration.

Now following the approach demonstrated in Example \ref{example3}, we show when a tuple $\alpha$ is an exact 2-typed parking function.

\begin{lemma}\label{lem:2color}
Let $M=m_1+m_2$. If $\alpha=(m_1;P_2)$ with $P_2=(p_1^{(2)},p_2^{(2)},\ldots,p_{m_2}^{(2)})$ satisfies $p_j^{(2)}\in\{0,1,2,\ldots, m_1\}$ for all $1\leq j\leq m_2$, then $\alpha \in \eTPF{M}{2}$.
\end{lemma}

\begin{proof}
    Let $\alpha=(m_1;P_2)$ where $P_2=(p_1^{(2)},p_2^{(2)},\ldots,p_{m_2}^{(2)})$ satisfies $p_j^{(2)}\in\{0,1,2,\ldots, m_1\}$ for all $1\leq j\leq m_2$. We then have $m_1$ cars of type 1 that park without preference. There remains $m_2$ cars of type 2 that each need to park with the preference of the $j^{\text{th}}$ car $p_j^{(2)}$, respectively. By Definition~\ref{def: exact kTPF}, we have that $\alpha$ is an exact $k$-TPF when $p_j^{(2)}$ is at most the number of type 1 cars. 
\end{proof}

    In following this approach, we see that all the cars have a preference value that matches a parking gap between $0$ and $m_1$ (inclusive).

The process used in the previous examples and lemma can be expanded for cases involving more than one preference list $P_i$. In these cases, for each preference list, we park cars with the same preference simultaneously. Now, we consider Lemma~\ref{lem:2color} for a general $k$, which is followed by the properties of permutations of a $k$-TPF $\alpha$. 

\begin{lemma}\label{lemma0.4} 
 Let $M$ and $(m_1, m_2,\ldots, m_k)$ be given. If all $P_i$ have preference values between $0$ and $\displaystyle{\sum_{j=1}^{i-1}m_j}$ for $i=2,3,\dots, k$, then $\alpha=(m_1;  P_2, P_3, \dots, P_k)$ is an exact $k$-TPF.
\end{lemma}

\begin{proof}
We prove this by induction on $i$. We have shown that the statement holds true for the base case $i=2$ from Lemma~\ref{lem:2color}. Now assume that $\alpha$ is an exact $k$-TPF for $2\leq i \leq k-1$. 
Then by induction for $i=k$, the cars of type $1, 2,\dots, k-1$ are all parked.
This means there is a total of $M-m_k$ cars parked; therefore there are $M-m_k+1$ gaps for the type $k$ cars to choose from.
Then, by Definition~\ref{def: exact kTPF}, $\alpha=(m_1;  P_2, P_3, \dots, P_k)$ is an exact $k$-TPF. 
\end{proof}

\begin{definition}(Parking permutation).\label{definition:permutation}
An exact $k$-TPF $\beta=(m_1;Q_2,\ldots,Q_k)$ is a \textit{parking permutation} of $\alpha=(m_1;P_2,\ldots,P_k)$ if there exist at least one $Q_i$ for $i=2,\ldots,k$ in $\beta$ that is a rearrangement of the corresponding $P_i$ in $\alpha$. 
\end{definition}

\begin{example}\label{ex:perm ex}
Let $\alpha=(4; (2,3,4))$. Here $M=7, \, m_1=4, \, m_2=3.$ Then all the parking permutations of the exact $k$-TPF $\alpha$ are
\[
    (4; (3,4,2)), \quad (4; (2,4,3)) \]
    \[\quad (4; (4,2,3)), \quad (4; (3,2,4)),\quad (4; (4,3,2)).
\]
 \end{example}

\begin{theorem}
\label{thrm1}
Every parking permutation of an exact $k$-TPF $\alpha$ of order $(m_1, m_2, \ldots, m_k)$ results in the same parking configuration $C$.
\end{theorem}

\begin{proof}
We begin by considering an exact 2-TPF $\alpha = (m_1; P_2)$ where $m_1$ is the number of cars of type 1 and $P_2=(p_1^{(2)}, p_2^{(2)}, \dots, p_{m_2}^{(2)})$ is the preference list for the $m_2$ cars of type 2. Now let $\beta = (m_1; P_2')$ be a parking permutation of an exact $k$-TPF $\alpha$, where $P_2'$ is some rearrangement of the elements in $P_2$. By definition, every car in $P_2'$ prefers exactly $p_j^{(2)}$ cars of type 1 parked before them for $1 \le j \le m_2$. Hence the cars are parked in each open gap depend only on their preference and not the order. Therefore, $\alpha$ and $\beta$ both result in the same configuration $C$. Thus, the base case holds. 

Now assume that for $k > 2$, every parking permutation of an exact $(k-1)$-TPF results in the same parking configuration. Consider an exact $k$-TPF $\alpha'=(m_1;P_2,\ldots,P_k)$ and let $\beta'=(m_1;Q_2,\ldots,Q_k)$ be a parking permutation of an exact $k$-TPF $\alpha$, meaning that for every $P_j$ with $j \in (2, \ldots, k)$, $Q_j$ is a parking permutation of $P_j$. We want to show that $\alpha$ and $\beta$ result in the same parking configuration. We can think of parking the $m_1$ cars of type 1 first. By the induction hypothesis, any parking permutation of the preferences for types 2 to $k-1$ will result in the same parking configuration.

Now, we park the $m_k$ cars of type $k$ according to their preference list for $\alpha$ and $\beta$ as done in our second technique. Similar to the base case, the final positions of the type $k$ cars depend only on the \emph{multiset} of preferences in $m_k$ for $\alpha$ and $\beta$, which are the same since $\alpha$ is a parking permutation of $\beta$. It is clear that the order in which the type $k$ cars choose their spots based on the preferences are permuted, but results in the same configuration.
\end{proof}

\begin{corollary}
    A given parking configuration $C$ can only arise in one of the following ways:
    \begin{itemize}
        \item an exact $k$-TPF $\alpha$;
        \item a parking permutation of $\alpha$;
        \item or both an exact $k$-TPF $\alpha$ and any one of its parking permutations.
    \end{itemize}
\end{corollary}

\begin{proof}
    We will proceed with proof by contradiction. Assume that there are two exact $k$-TPFs $\alpha$ and $\beta$ that are not parking permutations of one another that both result in parking configuration $C$. 
    
    Suppose $\alpha$ and $\beta$ are of different orders. Let $m_i^{\alpha}$ be the number of cars of type $i$ in $\alpha$ and $m_i^{\beta}$ be the number of cars of type $i$ in $\beta$. If we iterate through all cars of type $i=1,\ldots,k$, then there will be at least one case in which $m_i^{\alpha} \neq m_i^{\beta}$. As a result, the preference list for cars of type $i$ in $\alpha$ would not be equal to the preference list for cars of type $i$ in $\beta$. Therefore, since $\alpha$ and $\beta$ are not parking permutations of each other and there are different amounts of type $i$ cars in each parking function, $\alpha$ and $\beta$ could not result in the same parking configuration $C$.

    Now suppose $\alpha$ and $\beta$ have the same order.
    Then by definition, since $\alpha$ and $\beta$ are not parking permutations of one another, there must be no type $i$ car with the preference list $P_i$ for $\alpha$ in which is not a rearrangement of the preference list $Q_i$ for $\beta$.
    This means that the cars of type $i$ in $\alpha$ will park in a different preferred spot that is not a parking permutation of that pertaining to the cars of type $i$ in $\beta$. As a result, the parking configuration $C$ generated by $\alpha$ would not be the same as that generated by $\beta$.
\end{proof}

\section{Counting the Number of $k$-typed parking functions}\label{question1}

In this section we present a major result in counting exact $k$-typed parking functions. Our inspiration is drawn from Dukes' paper \cite{science-DOI-duke}, in which he presented the following questions:
\begin{center}
    \begin{itemize}
        \item How many 2-tiered parking functions are there?
        \item How many 3-tiered parking functions of order $(a,b,c)$ are there?
    \end{itemize}
\end{center}
In essence, by us focusing on exact $k$-typed parking functions, we aim to create a more concrete and manageable framework for exploring the combinatorial properties of tiered parking functions to follow. This directly allows us to analyze relationships between preferences and parking configurations in a more precise way. We begin the proof for the number of exact $k$-typed parking functions when $M$ is of order $(m_1, m_2, \ldots, m_k)$ by establishing a result which counts the number of exact 2-typed parking functions as a base case, then further generalize to give us a closed formula.

\begin{theorem}\label{thm:all k}
    Let $M \in \mathbb{N}$ and let $(m_1,\ldots,m_k)$ be a composition of $M$. There are 
    \begin{align}
    &\prod_{i=2}^{k}\left(1+\sum_{j=1}^{i-1}m_j\right)^{m_i}\label{eq:thm 1}
    \end{align}
exact $k$-TPFs of order $(m_1,\ldots,m_k)$.
\end{theorem}

\begin{proof}
An exact 2-TPF $\alpha = (m_1; P_2)$ order $(m_1, m_2)$ requires that all preference values $P_2$ for the $m_2$ cars of type 2 are between $0$ and $\sum_{j=1}^{2-1} m_j - 1 = m_1 - 1$. This means there are $m_1$ possible integer preference values for each of the $m_2$ cars (from $0$ to $m_1 - 1$ inclusive).
Now let's evaluate the formula from the claim for $k=2$: \begin{align*} \prod_{i=2}^{2}\left(1+\sum_{j=1}^{i-1}m_j\right)^{m_i} &= \left(1+\sum_{j=1}^{2-1}m_j\right)^{m_2} \ = \left(1+m_1\right)^{m_2} \end{align*} 
Now let's assume, by induction, that for some $s$ where $2 \le s < k$, the number of exact $s$-TPFs of order $(m_1, \dots, m_s)$ is given by 
\begin{align}
&\prod_{i=2}^{s}\left(1+\sum_{j=1}^{i-1}m_j\right)^{m_i}
\end{align}
and we want to show that for $s+1$, the number of exact $(s+1)$-TPFs is given by 
\begin{align}
&\prod_{i=2}^{s+1}\left(1+\sum_{j=1}^{i-1}m_j\right)^{m_i}
.\end{align}
Considering that the drivers of the cars of type $s+1$ each prefer to park after exactly some number of cars of types 1 through $k$, by definition, there are $\sum_{j=1}^{k}m_j$ cars of a lower type already parked. Thus, each car of type $s+1$ has $1+\sum_{j=1}^{k}m_j$ valid preference values. So, the number of ways to assign preferences to the $m_{s+1}$ cars of type $s+1$ is
\begin{align}
    \left(1+\sum_{j=1}^{k}m_j\right)^{m_{s+1}}
.\end{align}
If we recall the formula given for the exact $s$-TPFs for the first $s$ types, the total number of exact $(s+1)$-TPFs is 
\begin{align}
&\left[\prod_{i=2}^{s}\left(1+\sum_{j=1}^{i-1}m_j\right)^{m_i}\right] \times \left(1+\sum_{j=1}^{s}m_j\right)^{m_{s+1}} =
\prod_{i=2}^{s+1}\left(1+\sum_{j=1}^{i-1}m_j\right)^{m_i}
\end{align}

By the principle of induction the formula holds for all $s \ge 2$.
\end{proof}

\section{Cumulative Results on Mapping Families to Configurations}

Recall that in Section \ref{sect2}, we established that permuting the preference list entries of an exact $k$-typed parking function also results in an exact $k$-typed parking function. The initial parking function, along with its rearranged preference lists, then results in the same configuration $C$.

Consider the collection of parking permutations for a given parking function in which we formally define what it means to be a family. Our main result establishes that the size of each family is equivalent to the number of distinct resulting configurations.

\begin{definition}(Canonical Form).\label{definition:canonical form}
An exact $k$-TPF is in \textit{canonical form} if the preferences $p_j^{(i)}$ are listed in a weakly increasing order in each 
$P_i$ for $i=2,3,\ldots,k$.
\end{definition} 

\begin{example} The 4-TPF $\alpha = (4;(2,1,0,4,2), (8,2,1,2),(4,10,8))$ in its canonical form is $\alpha_{\text{canonical}}= (4,(0,1,2,2,4), (1,2,2,8), (4,8,10))$.
\end{example}

\begin{definition}(Family of $\alpha$).\label{definition:family}
Given an exact $k$-TPF $\alpha = (m_1; P_2,\ldots, P_k),$  we define \emph{family of $\alpha$}, which we denote fam$(\alpha)$, to be the set of all parking permutations of an exact $k$-TPF $\alpha$.
\end{definition}

Thus, as in Example \ref{ex:perm ex},

\begin{equation*}
\text{fam}(\alpha) = 
\left\{ 
    \begin {aligned}
    (4; (2,3,4)), \quad   (4; (3,2,4)), & \quad    (4; (4,2,3)),\\
 ( 4; (2,4,3)),\quad  (4; (3,4,2)),& \quad  (4; (4,3,2)) 
    \end{aligned}
\right\}.
\end{equation*}

The next result(s) gives the number of distinct parking configurations for each fam$(\alpha)$.

\begin{proposition}\label{prop2_1}
We fix $M$ and an order $(m_1,m_2,\dots,m_k)$. Every parking configuration $C$ corresponds to exactly one set fam($\alpha$), where $\alpha$ is some exact $k$-TPF of order $(m_1,m_2,\dots,m_k)$.
\end{proposition}

\begin{proof}
We begin by establishing the existence of a family given a configuration.
Consider a valid parking configuration $C$. 
We can find $m_1$ through $m_k$ by respectively counting the number of cars of types $1$ through $k$ that appear in $C$. 
In order to build $P_i$ for $i=2,\dots, k$, consider the available parking positions for cars of type $i$. These range from $0$ to $\mu_i = \displaystyle\sum_{j=1}^{i-1}m_j$ since each type $i$ car can park either before or after any car from a lower type based on their preference. 
$P_i$ is given by writing out the tuple $(p_1^{(i)},  p_2^{(i)}, \ldots, p_{m_i}^{(i)})$  where each entry of the tuple is given by listing out in (weakly) increasing order of the gaps $0$ to $\mu_i$ that each car is parked in. The result is a preference list $P_i$, and $\alpha$ in the form $\alpha = (m_1; P_2, \cdots, P_{k})$ has been constructed to generate the configuration $C$.
\begin{figure}[h]
    \centering
    \boxed{\textcolor{black}{2}}\boxed{\textcolor{black}{1}}\boxed{\textcolor{black}{1}}\boxed{\textcolor{black}{2}}\boxed{\textcolor{black}{1}}\boxed{\textcolor{black}{2}}\boxed{\textcolor{black}{1}}\boxed{\textcolor{black}{3}}\boxed{\textcolor{black}
    {1}}\boxed{\textcolor{black}
    {3}}\boxed{\textcolor{black}
    {1}}\boxed{\textcolor{black}
    {4}}\boxed{\textcolor{black}
    {3}}\boxed{\textcolor{black}
    {1}}\boxed{\textcolor{black}
    {4}}\\
    \vspace{0.5 cm}
    $\alpha = (7; (0,2,3), (7,8,10), (11,13))$
    \caption{Example of parking configuration \& an element in fam($\alpha$)}
    \label{fig:my_label}
\end{figure}

Now suppose the parking configuration $C$ is realized by fam($\alpha$) and fam($\gamma$), where $\text{fam}(\alpha) \neq \text{fam}(\gamma)$. Take $\alpha_i \in \text{fam}(\alpha)$ and $\gamma_i \in \text{fam}(\gamma)$ to be respectively parking permutations of the exact $k$-TPFs $\alpha$ and $\gamma$. By Theorem \ref{thrm1}, $\alpha_i$ and $\gamma_i$ must result in the same configuration $C$. Since each parking permutation results in the same configuration, $\text{fam}(\alpha)=\text{fam}(\gamma)$. This contradicts the assumption that fam($\alpha$) is not equivalent to fam($\gamma$), thus every configuration $C$ has one unique family. 
\end{proof}

\begin{corollary}\label{corollary3}
Let $M \in \mathbb{N}$ and let $\alpha$ be an exact $k$-TPF of order $(m_1, m_2, \ldots, m_k)$ be fixed. Then every fam($\alpha$) maps to exactly one parking configuration $C$.
\end{corollary}

\begin{proof}

We begin by taking fam$(\alpha)$ to be the collection of parking permutations of $\alpha=(m_1; P_2, P_3, \dots, P_k).$ 
Suppose for the sake of contradiction that there exists some parking permutations $\alpha_1, \alpha_2\in$ fam$(\alpha)$ such that $\alpha_1$ maps to the configuration $C_1$ and $\alpha_2$ maps to the configuration $C_2$, where $C_1\neq C_2.$ This is not possible since every element within a family must correspond to the same configuration. Hence, we have that $C_1=C_2$, and every fam($\alpha$) maps to a unique parking configuration $C$ as desired.
\end{proof}

\begin{lemma}\label{config_single_type}
Let $\mu_i=\displaystyle\sum_{j=1}^{i-1}m_j$.
Given a fixed $M$ and order $(m_1,m_2,\ldots,m_k)$, assume that the cars of type 1 through $i-1$ are already parked. Then there exists ${\mu_i+m_i \choose m_i}$ distinct ways to construct the preference list $P_i$ for 
$i\in \{2, \ldots, k\}.$
\end{lemma}

\begin{proof}
We will proceed by induction. By Proposition \ref{prop2_1} and Corollary \ref{corollary3}, we see that the distinct number of families of an exact $k$-TPF $\alpha$ can be given by the number of configurations $C$ for $m_1+m_2$. After cars of type 1 park, we have $m_1+1$ distinguishable gaps for the cars of type 2 to park, which become indistinguishable amongst each other. By invoking [\cite{science-DOI-ved}, Lemma 2.1], we have
\[{m_2+(m_1+1)-1 \choose (m_1+1)-1}= {m_1+m_2 \choose m_1}\]resulting configurations for an exact $2$-TPF $\alpha$ with $m_1$ cars parked. 

We now extend our result to the $i=k$ type, where the $m_i$ cars of type $i$ are indistinguishable.
There are ${\mu_i+1}$ distinguishable parking gaps available, and therefore, \[{m_i+(\mu_i + 1)-1 \choose (\mu_i+1)-1} = {m_i + \mu_i \choose \mu_i} = {\mu_{i+1} \choose \mu_i}\] resulting configurations possible for a given $i$ and fixed $k$.
\end{proof}

\noindent Thus, we have shown that each distinct fam($\alpha$) corresponds to exactly one configuration. We will proceed to show that the total number of resulting configurations for a fixed $M$ and order $(m_1, m_2, \ldots, m_k)$. 

\begin{theorem}\label{product_theorem}
The set of exact $k$-TPFs of order $(m_1,m_2,\ldots,m_k)$ result in
 \begin{align}
&L=\prod_{i=1}^k {m_1+\cdots + m_i \choose m_i}\label{eq:thm 2}
 \end{align}
distinct parking configurations.

\end{theorem}
\begin{proof}
    By Lemma \ref{config_single_type}, we established that for a fixed $k$, the $m_i$ cars of type $i$ have $\binom{m_i +\mu_i}{\mu_i}$ ways to park. As the preference list $P_i$ is given independently of cars of type less that $i$, we have $\prod_{i=1}^k {\sum_{j=1}^i m_j \choose m_i}$ many choices for $(m_1,P_2,...,P_k).$
\end{proof}

\section{Counting groupings of $k$-typed parking functions}

\indent Now we consider counting the amount of elements in each family given $M$ and $\alpha = (m_1, P_2, \ldots, P_k)$.
This allows us to develop another way of counting $k$-typed parking functions since families are distinct subsets of the total number of exact $k$-TPFs.
The sum of the amounts of $k$-typed parking functions among all families is the same as the total number of $k$-typed parking functions.
Before presenting these results, we must introduce the notion of a multiset and our modification of it, which simplifies the computations that must be done.

\begin{definition}(Multiset).\label{definition:multiset}
A \textit{multiset} is a set that allows for repeated values. 
Multisets can be written in a more condensed notation using ordered pairs in the form $\{(x_0, a_0), (x_1, a_1), \dots, (x_n, a_n)\}$, where each $a_i$ represents the number of $x_i$ in the set for $i=0,\,1,\dots, n$. 
\end{definition}

\subruninhead{Notation} The \textit{generative multiset of $\alpha$} will be denoted  $GM(\alpha)$ to represent the set $\{\mult(P_2),\, \mult(P_3), \,\dots, \mult(P_k)\}$. We also use $\mult(P_i)$ to be the multiset associated to $P_i$ such that \[\mult(P_i)=\{(0,a_{i_0}), (1, a_{i_1}), \dots, (\mu_i, a_{i_{\mu_i}})\}, \quad \mbox{ where } \displaystyle \mu_i=\sum_{j=1}^{i-1}m_j.\]

The first entry of each ordered pair $(j, a_{i_j})$ in $\mult(P_i)$ denotes the index of the parking gap, where $j$ ranges from $0$ to $\mu_i$. We get $\mu_i$ by considering the number of cars of type $1, 2, \ldots, i-1$ that have already parked before the cars of type $i$ start parking. 
As the order of an exact $k$-TPF $\alpha$ is $(m_1, m_2, \ldots, m_k)$, there are $m_1$ cars of type 1, $m_2$ cars of type 2, $\ldots$, and $m_{i-1}$ cars of type $i-1$ that are parked when cars of type $i$ start parking.
Therefore, when cars of type $i$ start to park there are $\displaystyle\sum_{j=1}^{i-1}m_j$ cars that are parked. This sum is denoted by $\mu_i$.

\begin{example}
Take $\alpha =(5; (2, 1, 3), \,(4, 0, 3)).$ 
Then the corresponding parking configuration is
\begin{align*}
    mult(P_2)&= \{(0,0), (1,1), (2, 1), (3, 1), (4, 0),(5, 0))\}\\
    &=\{(1,1), (2, 1), (3, 1)\}\\
    mult(P_3)&= \{(0,1), (1,0), (2, 0), (3, 1), (4, 1), (5,0), (6, 0), (7, 0), (8,0)\}\\
    &=\{(0,1), (3,1), (4,1)\}.
\end{align*}
So the generative multiset of this $\alpha$ is 
\[GM(\alpha) = \{\{(1,1), (2, 1), (3, 1)\}, \{(0,1), (3,1), (4,1)\}\}.\]
\end{example}

Having established the notation, we now prove a couple properties that connect families to generative multisets. This allows us to show a second way of counting all the $k$-TPF $\alpha$. 

\begin{lemma}
Given a $k$-TPF $\alpha=(m_1; P_2, P_3, \ldots, P_k)$ of order $(m_1,m_2,\dots,m_k)$, parking permutations of the same preference list $P_i$ result in the same generative multiset.
\end{lemma}

\begin{proof}
Consider the preference list $P_i = \left(p_{1}^{(i)}, \ldots, p_{m_i}^{(i)}\right)$ with associated multiset $\{(0, a_0), \ldots, (n, a_n)\}$.
A parking permutation $P_j = \left(p_{1}^{(j)}, \ldots, p_{m_j}^{(j)}\right)$ of $P_i$ has associated multiset $\{(0, b_0), \ldots, (n, b_n)\}$. The first entries in the multiset correspond to the indices of parking spots and the second indicate how many cars of type $i$ want to be parked at that index.

For each car of type $\{2,\dots, k\}$, the multiset has information on how many cars from each type $i$ are parked in each parking gap. 
Since $P_j$ is a parking permutation of $P_i$, it still contains the same parking preferences, although likely in a different order. 
This means that if $a_i$ cars from $P_i$ parked in gap $i$, then $a_i$ cars from $P_j$ parked in gap $i$ as well. 
This is true for all $0\leq i\leq n$, which means that $a_i = b_i \forall i$ in the range $0$ through $n$. Since all the respective multiset elements are the same, the multisets themselves are the same, and so parking permutations of the same list of preferences result in the same multiset.
\end{proof}

\begin{proposition}\label{fam_mult}
{Every fam($\alpha$) corresponds to exactly one generative multiset.} 
\end{proposition}

\begin{proof}
For existence, consider any $\alpha = (m_1; P_2,\dots, P_k)\in fam(\alpha)$. 

To form the generative multiset for a given family, each $P_i$ can be written as a multiset by counting the number of cars with each preference. 
Then form ordered pairs as $(j, a_{i_j})$, where $j$ corresponds to the index of the parking gap and $a_{i_j}$ corresponds with the number of cars in $P_i$ where parking gap $j$ is their preference.

As a result the multiset associated with $P_i$ is of the form \[mult(P_i) = \{(0, a_{i_0}), \ldots, (\mu_i, a_{i_{\mu_i}})\},\] where $\mu_i = \sum_{n=1}^{i-1} m_n$.

We do this for every $P_i$ and place the results into the set, \[GM(\alpha) = \{mult(P_2), mult(P_3), \ldots, mult(P_k)\},\]
which follows from the definition of the generative multiset of $\alpha$.

In order to prove uniqueness, assume by way of contradiction that fam($\alpha$) corresponds to more than one generative multiset. 
Specifically, let fam($\alpha$) correspond to both $GM(\alpha_1) = \{\mult(P_2^1), \mult(P_3^1), \ldots, \mult(P_k^1)\}$ and $GM(\alpha_2) = \{mult(P_2^2), mult(P_3^2), \ldots, mult(P_j^2)\}$, where $GM(\alpha_1) \neq GM(\alpha_2)$. 
First note that since $\alpha$ has $k$ elements for an exact $k$-TPF, $GM(\alpha_1)$ and $GM(\alpha_2)$ must both have $k$ elements as well.
$GM(\alpha_1)$ and $GM(\alpha_2)$ are sets, and since we know that $|GM(\alpha_1)| = |GM(\alpha_2)|$, at least one of the corresponding $mult(P_i)$ elements from $GM(\alpha_1)$ and $GM(\alpha_2)$ must be such that $mult(P_i^1) \neq mult(P_i^2)$. 
This leads to a contradiction because fam($\alpha$) is by definition the set containing all parking permutations of an exact $k$-TPF $\alpha$. 
It was previously shown that parking permutations of the same list of preferences result in the same multiset, and so the preference lists associated with $mult(P_i^1)$ and $mult(P_i^2)$ must be parking permutations of one another since they take into account the same type, which means that their associated multisets must be the same.
\end{proof}

\begin{proposition} \label{mult_fam}
Given an exact $k$-TPF $\alpha$ of order $(m_1,m_2,\dots,m_k)$, every generative multiset of $\alpha$ corresponds to exactly one fam($\alpha$). \label{prop2_2} 
\end{proposition}

\begin{proof}
We begin by recalling that the set of all parking permutations of an exact $k$-TPF $\alpha$ result in fam($\alpha$). So observe that for a given $GM(\alpha)$, we can recover its associated $\alpha$. Let $\alpha=(m_1; P_2, \ldots, P_k)$ and $GM(\alpha)$ be the generative multiset of $\alpha$. For the sake of contradiction, suppose we construct $\beta=(m_1; Q_2, \ldots, Q_k) \notin$ fam$(\alpha)$ from $GM(\alpha)$. 

This is the same as mapping $GM(\alpha)$ to $\text{fam}(\beta)$, where the map is given by taking the collection of parking permutations of $\alpha.$ In other words, $GM(\alpha)=GM(\beta).$ This contradicts that $\text{fam}(\beta)\neq\text{fam}(\alpha)$ by $\beta\notin \text{fam}(\alpha)$, therefore $\beta$ is a parking permutation of $\alpha$.
\end{proof}

\begin{corollary}
    There is a one to one correspondence between $GM(\alpha)$ and $\text{fam}(\alpha)$ for a given $\alpha$ of order $(m_1, m_2, \ldots, m_k)$.
\end{corollary}

\begin{proof}
    This follows immediately from combining Proposition~\ref{fam_mult} and Proposition~\ref{mult_fam}.
\end{proof}

\begin{theorem}
\label{num_elements_in_fam}
For a fixed $M$ and $\alpha=(m_1;P_2,\ldots,P_k)$ of order $(m_1, m_2, \ldots, m_k), $ the number of elements in each fam$(\alpha)$ is given by
 \begin{align}
&F(\alpha)=\prod_{i=1}^{k}\frac{m_i!}{a_{i_0}!a_{i_1}!\cdots a_{i_{\mu_i}}!}\label{eq:thm 3}
 \end{align}
\end{theorem}

\begin{proof}
Let $\alpha = (m_1; P_2, \ldots, P_k)$ in its canonical form. 
The generative multiset associated with $\alpha$ is $GM(\alpha)=\{mult(P_2),\, mult(P_3), \,\dots, mult(P_k)\}$.
That is, $\text{mult}(P_i)=\{(0,a_{i_0},), (1, a_{i_1}), \ldots, (\mu_i, a_{i_{\mu_i}})\}$ and $\mu_i=\displaystyle\sum_{j=1}^{i-1}m_j$.
Using the formula for parking permutations with repeated elements, the number of parking permutations of $P_i$ is represented by  \[\frac{m_i!}{a_{i_0}!a_{i_1}!\cdots a_{i_{\mu_i}}!}.\]

When cars of type $i$ park, the cars of type $\{1, 2, \ldots, i-1\}$ are indistinguishable to cars of type $i$.
Therefore, we can say the way we park each type of car is an independent event.
So we multiply the total number of ways to park for each type to get 
\[\prod_{i=1}^{k}\frac{m_i!}{a_{i_0}!a_{i_1}!\cdots a_{i_{\mu_i}}!}.\]
This is the number of parking permutations for all $P_i$'s for a given value of $i$.
Then by Proposition \ref{prop2_2}, we know every generative multiset corresponds to exactly one fam$(\alpha)$.
Therefore the product above is the number of elements in a given fam$(\alpha)$.
\end{proof}

\begin{remark}
Note that if $i = 1$, there is only one way to park cars on type $1$. Namely, we have $\frac{m_1!}{m_1!}=1.$
\end{remark}

\begin{theorem}
The number of distinct $k$-TPFs is given by the formula
 \begin{align}
&\sum_{j=1}^{L} F(\alpha_j)\label{eq:thrm5.9},\end{align}
where $j$ is an index over each distinct fam($\alpha$), $L$ is as in Theorem \ref{product_theorem}, and F is defined in equation (7).
\end{theorem}

\begin{proof}

We can index the families of an exact $k$-TPF $\alpha$ from 1 through $L$, since by Theorem \ref{product_theorem} there are $L$ distinct parking configurations, which is the same as the number of fam($\alpha$) by Proposition \ref{prop2_1}. 
For each of the families, from Theorem \ref{num_elements_in_fam} there are \[\prod_{i=1}^{k}\frac{m_i!}{a_{i_0}!a_{i_1}!\cdots a_{i_{\mu_i}}!}.\] $k$-TPFs that are contained within that family. By counting the number of $k$-TPFs in each family for all of the $L$ families, we count the total number of distinct $k$-TPFs, which is \[\sum_{j=1}^{L} F(\alpha_j)\]$k$-TPFs.
\end{proof}

We observe that equation \eqref{eq:thm 1} must equal equation \eqref{eq:thm 3}, i.e.
 \begin{align}
\prod_{i=2}^{k}\left(1+\sum_{j=1}^{i-1}m_j\right)^{m_i}&=\sum_{j=1}^{L} F(\alpha_j)
 \end{align}
where $L=\displaystyle\prod_{i=1}^k {\sum_{j=1}^i m_j \choose m_i}$ and $\alpha$ is a $k$-typed parking function of order $(m_1, m_2, \ldots, m_k)$ for a fixed $M$.

\section{Discussion and Future Works}
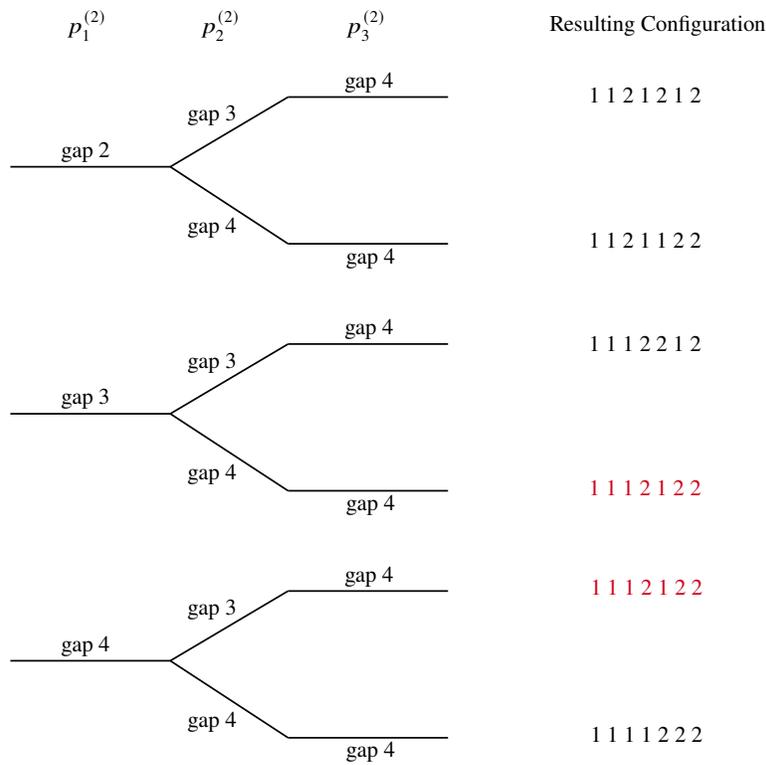
\begin{figure}[h!!!!!!]
    \centering
    \scalebox{1}
    {\input{figures/futureworks}}
    \caption{Parking cars using ``at-least''} definition
    \label{fig:my_label}
\end{figure}
One may have noticed that setting $m_1=n$ and $m_2=n-1$ in the base case of Theorem \ref{thm:all k} will recover the enumeration for the number of classical parking functions as described in the introduction. This encouraged us to ask, ``What are the algebraic implications between $k$-typed parking functions and classical parking functions?'' 
Considering the substantial research that has already been made about the connections between revised definitions of parking functions to fundamental, combinatorial objects such as Catalan numbers and Dyck Paths \cite{science-DOI-shuffle}, we can inquire about the role exact $k$-TPFs can play.
Another point is that using our definition of $k$-typed parking functions, we repeat Dukes'~\cite{science-DOI-duke} question using the terminology developed in our paper: How many $2$-typed parking functions are there? How many $3$-typed parking functions of order $(a,b,c)$ are there? We observe that it is challenging to count the number of distinct configurations, using the original definition of $k$-typed parking functions, by providing an example and ending with our own open question. 

Recall the original $k$-TPF definition:
\begin{quote}
For each preference list $P_i$, the driver of the $j^{\rm th}$ car has the preference that there be \textbf{at least} $p_j^{(i)}$ cars of types $\{1, \ldots, i-1\}$ parked before them.
\end{quote}
We begin with the case where $k=2$ given an exact $k$-TPF $\alpha = (4;(2,3,4))$. Each element in the preference list $P_2$ can park 
between gaps $p_j^{(2)}$ and $m_1$ (inclusive). That is: gap 2 $\leq p_1^{(2)} \leq$ gap 4, gap 3 $\leq p_2^{(2)} \leq$ gap 4, and $p_3^{(2)} \leq$ gap 4. 
Observe that in Figure \ref{fig:my_label}, there are two identical configurations resulting from the same preference list. This is because multiple cars within the same preference list may want to park in the same gap. 
\begin{quote}
\textbf{Problem 1:} Given $\alpha = (m_1;(a_1,\ldots,a_\ell))$, how many distinct resulting configurations are there using the at-least definition?
\end{quote}

\input{references1}

\begin{acknowledgement}
The authors would like to thank the entire SLMath/MSRI community for their continued support and useful advice. Special thank you to Professor Rebecca Garcia and Professor Pamela E. Harris for presenting these problems to us and for their mentorship. Thank you to all professors and colleagues from our respective institutions that provided help in any capacity. 
\end{acknowledgement}

\ethics{Competing Interests}{\newline
We acknowledge that Lina Liu is supported by the NSF GRFP Grant 2237827. The authors have no conflicts of interest to declare that are relevant to the content of this chapter.}

\end{document}

%% file: figures/futureworks.tex
\tikzset{every picture/.style={line width=0.65pt}} 

\begin{tikzpicture}[x=0.55pt,y=0.55pt,yscale=-1,xscale=1]

\draw    (49,118) -- (159,118) ;
\draw    (159,118) -- (240,70) ;
\draw    (159,118) -- (240,171) ;
\draw    (240,70) -- (350,70) ;
\draw    (240,171) -- (350,171) ;

\draw    (49,288) -- (159,288) ;
\draw    (159,288) -- (240,240) ;
\draw    (159,288) -- (240,341) ;
\draw    (240,240) -- (350,240) ;
\draw    (240,341) -- (350,341) ;
\draw    (49,458) -- (159,458) ;
\draw    (159,458) -- (240,410) ;
\draw    (159,458) -- (240,511) ;
\draw    (240,410) -- (350,410) ;
\draw    (240,511) -- (350,511) ;

\draw (82,99) node [anchor=north west][inner sep=0.75pt]   [align=left] {gap 2\\};
\draw (169,74) node [anchor=north west][inner sep=0.75pt]   [align=left] {gap 3\\};
\draw (277,51) node [anchor=north west][inner sep=0.75pt]   [align=left] {gap 4\\};
\draw (278,172) node [anchor=north west][inner sep=0.75pt]   [align=left] {gap 4\\};
\draw (169,151) node [anchor=north west][inner sep=0.75pt]   [align=left] {gap 4\\};
\draw (169,491) node [anchor=north west][inner sep=0.75pt]   [align=left] {gap 4\\};
\draw (278,512) node [anchor=north west][inner sep=0.75pt]   [align=left] {gap 4\\};
\draw (277,391) node [anchor=north west][inner sep=0.75pt]   [align=left] {gap 4\\};
\draw (169,414) node [anchor=north west][inner sep=0.75pt]   [align=left] {gap 3\\};
\draw (82,439) node [anchor=north west][inner sep=0.75pt]   [align=left] {gap 4\\};
\draw (418,12) node [anchor=north west][inner sep=0.75pt]   [align=left] {Resulting Configuration\\};
\draw (445,61.4) node [anchor=north west][inner sep=0.75pt]    {$1\ 1\ 2\ 1\ 2\ 1\ 2\ $};
\draw (445,161.4) node [anchor=north west][inner sep=0.75pt]    {$1\ 1\ 2\ 1\ 1\ 2\ 2\ $};
\draw (445,232.4) node [anchor=north west][inner sep=0.75pt]    {$1\ 1\ 1\ 2\ 2\ 1\ 2\ $};
\draw (445,331.4) node [anchor=north west][inner sep=0.75pt]  [color={rgb, 255:red, 208; green, 2; blue, 27 }  ,opacity=1 ]  {$1\ 1\ 1\ 2\ 1\ 2\ 2\ $};
\draw (446,400.4) node [anchor=north west][inner sep=0.75pt]  [color={rgb, 255:red, 208; green, 2; blue, 27 }  ,opacity=1 ]  {$1\ 1\ 1\ 2\ 1\ 2\ 2\ $};
\draw (446,501.4) node [anchor=north west][inner sep=0.75pt]    {$1\ 1\ 1\ 1\ 2\ 2\ 2\ $};
\draw (86,8.4) node [anchor=north west][inner sep=0.75pt]    {$p_{1}^{( 2)}$};
\draw (178,8.4) node [anchor=north west][inner sep=0.75pt]    {$p_{2}^{( 2)}$};
\draw (278,8.4) node [anchor=north west][inner sep=0.75pt]    {$p_{3}^{( 2)}$};
\draw (82,269) node [anchor=north west][inner sep=0.75pt]   [align=left] {gap 3\\};
\draw (169,244) node [anchor=north west][inner sep=0.75pt]   [align=left] {gap 3\\};
\draw (277,221) node [anchor=north west][inner sep=0.75pt]   [align=left] {gap 4\\};
\draw (278,342) node [anchor=north west][inner sep=0.75pt]   [align=left] {gap 4\\};
\draw (169,321) node [anchor=north west][inner sep=0.75pt]   [align=left] {gap 4\\};

\end{tikzpicture}

%% file: references1.tex
%
%
%